\documentclass[11pt]{article} 

\usepackage[utf8]{inputenc} 
  
\usepackage{geometry} 
\usepackage{graphicx} 

\usepackage{booktabs} 
\usepackage{array} 
\usepackage{paralist} 
\usepackage{verbatim} 
\usepackage{amsmath}
\usepackage{amssymb}
\usepackage{amsthm}
\usepackage{xifthen}
\usepackage{xcolor}

\usepackage{subfigure}

\usepackage[colorlinks=true,linkcolor=blue,urlcolor=blue]{hyperref}

\usepackage{fancyhdr} 
\usepackage{sectsty}
\allsectionsfont{\sffamily\mdseries\upshape} 

\newtheorem{theorem}{Theorem}

\newtheorem{lemma}[theorem]{Lemma}

\newtheorem{proposition}[theorem]{Proposition}
\newtheorem{remark}[theorem]{Remark}

\usepackage[nottoc,notlof,notlot]{tocbibind} 
\usepackage[titles,subfigure]{tocloft} 

\newcommand{\E}{\mathbb E}

\newcommand{\F}{\mathcal F}

\newcommand{\R}{\mathbb R}
\renewcommand{\P}{\mathbb P}

\newcommand{\kom}[1]{}
\renewcommand{\kom}[1]{{\bf [#1]}}

\newcounter{komcounter}
\numberwithin{komcounter}{section}

\title{
Stopping problems with an unknown state}
\author{Erik Ekström\\\textit{Department of Mathematics, Uppsala University}\\ \\
Yuqiong Wang\\\textit{Department of Mathematics, Uppsala University}\\
}

\begin{document}

\maketitle

\begin{abstract}
We extend the classical setting of an optimal stopping problem under full information to include for problems
with an unknown state. The framework allows the unknown state to influence (i) the drift of the underlying process, (ii) the payoff functions, and (iii) the distribution of the time horizon. 
Since the stopper is assumed to observe the underlying process and the random horizon, 
this is a two-source learning problem. Assigning a prior distribution for the 
unknown state, filtering theory can be used to embed the problem in a Markovian framework, and we thus reduce the problem
with incomplete information to a problem with complete information but with one more state-variable.
We provide a convenient formulation of the reduced problem, based on a measure change technique that decouples 
the underlying process from the state variable representing the posterior of the unknown state. 
Moreover, we show by means of several new 
examples that this reduced formulation can be used to solve problems explicitly.
\end{abstract}

\section{Introduction}\label{intro}

In most literature on optimal stopping theory, the stopper acts under full information about the 
underlying system. In some applications, however, information is limited, and the stopper 
then needs to base her decision only on the information available upon stopping.
We study stopping problems of the type
\begin{equation}
\label{osp1}
\sup_{\tau}\E\left[g(\tau,X_\tau,\theta)1_{\{\tau< \gamma\}}+ h(\gamma,X_\gamma,\theta)1_{\{ \gamma\leq\tau\}}\right],
\end{equation}
where $X$ is a diffusion process; here
$g$ and $h$ are given functions representing the payoff if stopping occurs before and after the random time horizon $\gamma$, respectively, and $\theta$ is a Bernoulli random variable representing
the unknown state. This unknown state may influence the drift of the diffusion process $X$, the distribution of the random horizon $\gamma$ and the payoff functions $g$ and $h$. 

Cases with $g(t,x,\theta)=g(t,\theta)$ and $h(t,x,\theta)=h(t,\theta)$ are closely related to statistical problems, where the process $X$ merely serves as an observation process but does not affect the payoff upon stopping. A classical example is the sequential testing problem for a Wiener process, see \cite{S} for a perpetual version and \cite{NP} for a version with a random 
horizon. 
Cases with $g(t,x,\theta)=g(t,x)$ and $h(t,x,\theta)=h(t,x)$, on the other hand, where the unknown state does not affect the payoff directly but only implicitly via the dynamics of $X$, 
have been studied mainly in the financial literature. For example, American options with
incomplete information about the drift of the underlying process have been studied in 
\cite{DMV} and \cite{EVa}, and a liquidation problem has been studied in \cite{EL}.
Related literature include studies of models containing change-points, see \cite{G} and \cite{HKMR}, 
a study allowing for an arbitrary distribution of the unknown state (\cite{EV}), 
problems of stochastic control (\cite{L}) and singular control (\cite{da}), and stochastic games (\cite{DGV}) under incomplete information.
Stopping problems with a random time horizon are studied in, for example, \cite{CG} and \cite{LM}, where the authors consider models with a random finite time horizon independent of the underlying process. 

In the current article, we study the optimal stopping problem using the general formulation in \eqref{osp1}, which  
is flexible enough to accommodate several new examples. In particular, the notion of a state-dependent
random horizon appears to be largely unstudied, even though it is a  natural ingredient 
in many applications. Indeed, consider a situation
where the unknown state is either "good" ($\theta=1$) or "bad" ($\theta=0$) for an agent who is thinking of investing
in a certain business opportunity. Since agents are typically subject to competition, the business opportunity would 
eventually disappear, and the rate with which it does so would typically be larger in the "good" state than in the "bad" state. The disappearance of a business opportunity is incorporated in our set-up by choosing the compensation $h\equiv 0$.

In some applications, it is more natural to have a 
random state-dependent horizon at which the stopper is forced to stop (as opposed to missing out on the opportunity). For example, 
in modeling of financial contracts with recall risk (see, e.g., \cite{GH}), the party who makes the recall would decide on a time point at which the positions at hand have to be terminated. 
Consequently, problems with $h=g$ can be viewed as problems of forced stopping. 
More generally, the random horizon can be useful in models with competition, 
where $h\leq g$ corresponds to situations with first-mover advantage, and $h\geq g$ to situations with second-mover advantage.

We first apply filtering methods to the stopping problem \eqref{osp1}, which 
allows us to re-formulate the stopping problem in terms of a two-dimensional state process $(X,\Pi)$, where 
$\Pi$ is the probability of one of the states conditional on observations. Then 
a measure change technique is employed, where the dynamics of the diffusion process $X$ under the new measure
are unaffected by the 
unknown state, whereas the Radon-Nikodym derivative can be fully expressed in terms of $\Pi$.
Finally, it is shown how the general set-up, with two spatial dimensions, can be reduced further 
in specific examples. In fact, we provide three different examples (a hiring problem, a problem of optimally closing a short position, and a sequential testing problem with random horizon) where it turns out that
the spatial dimension is one-dimensional so that the problems can be fully analysed. The examples are mainly
of motivational character, and in order not to burden the presentation with too many details, we content ourselves with 
providing the reduction to one spatial dimension -- a detailed study of the corresponding one-dimensional problem
can then be performed using standard methods of optimal stopping theory.

\section{Problem specification}

We consider a Bayesian set-up where one observes a diffusion process $X$ in continuous time, 
the drift of which depends on an unknown state $\theta$ that takes values 0 and 1 with probabilities
$1-\pi$ and $\pi$, respectively. Given payoff functions $g$ and $h$, the problem is to 
stop the process so as to maximize the expected reward in \eqref{osp1}. Here the random horizon
$\gamma$ has a state-dependent distribution, but is independent of the noise of $X$.

The above set-up can be realised by considering a
probability space $(\Omega, \F,\P_\pi)$ hosting a standard Brownian motion $W$ and 
an independent Bernoulli-distributed random variable $\theta$ with $\P_\pi(\theta=1)=\pi=1-\P_\pi(\theta=0)$.
Additionally, we let $\gamma$ be a random time (possible infinite) independent of $W$ and with state-dependent survival distribution
\[\P_\pi(\gamma>t\vert \theta=i) = F_i(t),\]
where $F_i$ is continuously differentiable and non-increasing with $F_i(0)=1$ and $F_i(t)>0$ for all $t\geq 0$, $i=0,1$. We remark that we include the possibility that $F_i\equiv 1$ for some $i\in\{0,1\}$
(or for both), corresponding to an infinite horizon.
We then have 
\begin{equation}
\label{Ppi}
\mathbb P_\pi=(1-\pi)\mathbb P_0 + \pi\mathbb P_1,
\end{equation} 
where $\mathbb P_i$ is a probability measure under which $\theta=i$, $i=0,1$. Now let
\begin{equation}
\label{X}
dX_t= \mu(X_t,\theta) \,dt + \sigma(X_t) \,dW_t,
\end{equation}
which on each event $\{\theta=i\}$ is an Ito diffusion with drift $\mu( \cdot,i)$.
Here $\mu(\cdot, \cdot): \R\times\{0,1\}\to\R$  is a given function of the unknown state $\theta$ and the current value of the underlying process; we denote by $\mu_0(x)=\mu(x,0)$ and $\mu_1(x)=\mu(x,1)$. 
The diffusion coefficient $\sigma(\cdot): \R\to (0,\infty)$ is a given function of $x$, independent of the unknown state $\theta$. We assume that the functions $\mu_0,\mu_1$ and $\sigma$ satisfy standard Lipschitz conditions so that the existence and uniqueness of a strong solution $X$ is guaranteed. 
We are also given two functions
$g(\cdot,\cdot,\cdot):[0,\infty)\times\R\times\{0,1\}\to\R$ and $h(\cdot,\cdot,\cdot):[0,\infty)\times\R\times\{0,1\}\to\R$, which we refer to as the 
payoff functions. We will sometimes use the notation $g_i(\cdot,\cdot):=g(\cdot,\cdot,i)$ and $h_i(\cdot,\cdot):=h(\cdot,\cdot,i)$ to 
denote the payoff functions on the event $\{\theta=i\}$, $i=0,1$, and we assume that $g_i$ and $h_i$ are continuous for $i=0,1$.

Denote by $\F^X$ the smallest right-continuous filtration that makes $X$ adapted,
and let $\mathcal T^X$ be the set of $\F^X$-stopping times. Similarly, denote by 
$\F^{X,\gamma}$ the smallest right-continuous filtration to which both $X$ and the process
$1_{\{\cdot\geq \gamma\}}$ are adapted, and let $\mathcal T^{X,\gamma}$ be the set of $\F^{X,\gamma}$-stopping times.

We now consider the optimal stopping problem
\begin{equation}
\label{osp}
V= \sup_{\tau\in\mathcal T^{X,\gamma}}\E_\pi\left[g(\tau,X_\tau,\theta)1_{\{\tau< \gamma\}}
+ h(\gamma,X_\gamma,\theta)1_{\{ \gamma\leq\tau\}}\right].
\end{equation}
In \eqref{osp}, and in similar expressions throughout the paper, we use the 
convention that $h(\tau,X_\tau,\theta):=0$ on the event $\{\tau=\gamma=\infty\}$. We further assume that the integrability condition
\[\E_\pi\left[\sup_{t\geq0}\left\{ \vert g(t,X_t,\theta)\vert + \vert h(t,X_t,\theta)\vert\right\}\right]<\infty\]
holds.

\begin{remark}
The unknown state $\theta$ in the stopping problem \eqref{osp} influences 
\begin{itemize}
\item[(i)] the drift of the process $X$,
\item[(ii)] the payoffs $g$ and $h$, and 
\item[(iii)] the survival distribution of the random horizon $\gamma$.
\end{itemize}
More precisely, on the event $\{\theta=0\}$ the drift of $X$ is $\mu_0(\cdot)$, the payoff functions are $g_0(\cdot,\cdot)$ and $h_0(\cdot,\cdot)$, and the random horizon has survival distribution function $F_0(\cdot)$; on the event $\{\theta=1\}$, 
the drift is $\mu_1$, the payoff functions are $g_1(\cdot,\cdot)$  and $h_1(\cdot,\cdot)$, and the random horizon has survival distribution function $F_1(\cdot)$.
\end{remark}

\section{Reformulation of the problem using filtering theory}

In this section we rewrite the optimal stopping problem \eqref{osp} with incomplete information as an optimal 
stopping problem with respect to stopping times in $\mathcal T^X$ and with complete information.

First consider the stopping problem
\begin{equation}
\label{osp2}
\hat V= \sup_{\tau\in\mathcal T^X}\E_\pi\left[g(\tau,X_\tau,\theta)1_{\{\tau< \gamma\}}
+ h(\gamma,X_\gamma,\theta)1_{\{\gamma\leq\tau\}}\right],
\end{equation}
where the supremum is taken over $\mathcal F^X$-stopping times. Since $\mathcal T^X\subseteq\mathcal T^{X,\gamma}$, we 
have $\hat V\leq V$. 
On the other hand, by a standard argument, cf. \cite{CG} or \cite{LM}, we also have the reverse inequality, so 
$\hat V=V$. Indeed, first recall that for any 
$\tau\in\mathcal T^{X,\gamma}$ there exists $\tau^\prime\in\mathcal T^X$ such that
$\tau\wedge\gamma = \tau^\prime\wedge\gamma$, see \cite[page 378]{Pro}. 
Consequently, $\tau=\tau'$ on $\{\tau<\gamma\}=\{\tau'<\gamma\}$ and $\tau\wedge\gamma=\tau'\wedge\gamma=\gamma$ on 
$\{\tau\geq \gamma\}=\{\tau'\geq\gamma\}$, so
\begin{eqnarray*}
\E_\pi\left[g(\tau,X_\tau,\theta)1_{\{\tau< \gamma\}}+ h(\gamma,X_\gamma,\theta)1_{\{\gamma\leq\tau\}}\right] =
 \E_\pi\left[g(\tau^\prime,X_{\tau^\prime},\theta)1_{\{\tau^\prime< \gamma\}} + h(\gamma,X_\gamma,\theta)1_{\{\gamma\leq\tau'\}}\right],
\end{eqnarray*}
from which $\hat V=V$ follows. Moreover, if $\tau'\in\mathcal T^X$ is optimal in \eqref{osp2}, then it is also optimal in \eqref{osp}.

\begin{remark}
Since we assume that the survival distributions are continuous, we have
\[\P_\pi(\tau=\gamma<\infty)=0\] for any $\tau\in\mathcal T^X$. Consequently, we can alternatively write
\[\hat V=\sup_{\tau\in\mathcal T^X}\E_\pi\left[g(\tau,X_\tau,\theta)1_{\{\tau\leq \gamma\}}
+h(\gamma,X_\gamma,\theta)1_{\{\gamma<\tau\}}\right].\]
\end{remark}

To study the stopping problem \eqref{osp}, or equivalently, the optimal stopping problem \eqref{osp2}, we introduce the conditional probability process
\[\Pi_t:=\P_\pi(\theta=1\vert\F^X_t)\] 
and the corresponding probability ratio process
\begin{equation}
\label{Phi}
\Phi_t:=\frac{\Pi_t}{1-\Pi_t}.
\end{equation}
Note that $\Pi_0=\pi$ and $\Phi_0=\varphi:=\pi/(1-\pi)$, $\P_\pi-$a.s.

\begin{proposition}
We have
\begin{eqnarray}
\label{Vnew}
V &=& \sup_{\tau\in\mathcal T^X}\E_\pi\Big[ g_0(\tau,X_\tau)(1-\Pi_\tau)F_0(\tau) + g_1(\tau,X_\tau)\Pi_\tau F_1(\tau)\\
\notag
&&\hspace{15mm}- \int_0^\tau \left( h_0(t,X_t)(1-\Pi_t)F'_0(t) + h_1(t,X_t)\Pi_t F'_1(t)\right)dt
\Big].
\end{eqnarray}
Moreover, if $\tau\in\mathcal T^X$ is optimal in \eqref{Vnew}, then it is also optimal in \eqref{osp}.
\end{proposition}

\begin{proof}
Denoting by $\tau$ a stopping time in $\mathcal T^X$, 
the tower property yields
\begin{eqnarray*}
\E_\pi\left[g(\tau,X_\tau,\theta)1_{\{\tau< \gamma\}}\right] &=& 
\E_\pi\left[ 1_{\{\tau< \gamma\}}\E_\pi\left[g(\tau,X_\tau,\theta)\vert \F^{X,\gamma}_\tau\right]\right],
\end{eqnarray*}
and we note that
\begin{eqnarray*}
1_{\{\tau< \gamma\}}\E_\pi\left[g(\tau,X_\tau, \theta)\vert \F^{X,\gamma}_\tau\right]
\hspace{-25mm} && \hspace{17mm} = 1_{\{\tau< \gamma\}}\left(g_0(\tau,X_\tau)\P_\pi(\theta=0\vert\F^{X,\gamma}_\tau) + 
g_1(\tau,X_\tau)\P_\pi(\theta=1\vert\F^{X,\gamma}_\tau)\right)\\
&=& \frac{1_{\{\tau< \gamma\}}}{\P_\pi(\tau<\gamma\vert\F^X_\tau)}\left( g_0(\tau,X_\tau)\P_\pi(\theta=0, \tau< \gamma\vert \F^X_\tau) + 
g_1(\tau,X_\tau)\P_\pi(\theta=1, \tau< \gamma\vert \F^X_\tau)\right)\\
&=& 1_{\{\tau< \gamma\}}\frac{ g_0(\tau,X_\tau)(1-\Pi_\tau)F_0(\tau) + g_1(\tau,X_\tau)\Pi_\tau F_1(\tau) }{\P_\pi(\tau< \gamma\vert\F^X_\tau)}.
\end{eqnarray*}
Another use of the tower property thus yields
\begin{eqnarray*}
\E_\pi\left[g(\tau,X_\tau,\theta)1_{\{\tau< \gamma\}}\right] &=& 
\E_\pi\left[ \E_\pi\left[1_{\{\tau< \gamma\}}\vert\F^X_\tau\right]\frac{ g_0(\tau,X_\tau)(1-\Pi_\tau)F_0(\tau) + g_1(\tau,X_\tau)\Pi_\tau F_1(\tau) }{ \P(\tau< \gamma\vert\F^X_\tau) } \right]\\
&=& \E_\pi\left[g_0(\tau,X_\tau)(1-\Pi_\tau)F_0(\tau) + g_1(\tau,X_\tau)\Pi_\tau F_1(\tau) \right].
\end{eqnarray*}

For the second term, we have that
\begin{eqnarray*}
 \E_\pi\left[ h(\gamma,X_\gamma,\theta)1_{\{\gamma\leq\tau\}}\right]
&=& \int_0^\infty \E_\pi\left[ h(\gamma,X_\gamma,\theta)1_{\{\gamma\leq\tau\}}1_{\{\gamma\in (t,t+dt)\}}
\right]\\
&=& \int_0^\infty \E_\pi\left[ \E_\pi\left[h(t,X_t,\theta)1_{\{t\leq\tau\}}1_{\{\gamma\in (t,t+dt)\}}\vert \F_t^X
\right]\right]\\
&=& \E_\pi\left[ \int_0^\infty  1_{\{t\leq\tau\}}\E_\pi\left[h(t,X_t,\theta)1_{\{\gamma\in (t,t+dt)\}}\vert \F_t^X
\right]\right]\\
&=& -\E_\pi\left[ \int_0^\tau   \left(h_0(t,X_t)(1-\Pi_t)F_0'(t) +h_1(t,X_t)\Pi_tF_1'(t)\right)dt \right].
\end{eqnarray*}

The optimal stopping problem \eqref{osp} therefore coincides with the stopping problem
\begin{eqnarray*}
&& \sup_{\tau\in\mathcal T^X}\E_\pi\Big[ g_0(\tau,X_\tau)(1-\Pi_\tau)F_0(\tau) + g_1(\tau,X_\tau)\Pi_\tau F_1(\tau) \\
&& \hspace{20mm}- \int_0^\tau \left( h_0(t,X_t)(1-\Pi_t)F'_0(t) + h_1(t,X_t)\Pi_t F'_1(t)\right)dt\Big].
\end{eqnarray*}
\end{proof}

From filtering theory, see e.g.  \cite{BC} and \cite{LS}, it is well-known that the pair $(X,\Pi)$ satisfies 
\begin{equation}
\label{XPi}
\left\{\begin{array}{ll}
dX_t=(\mu_0(X_t)+ (\mu_1(X_t)-\mu_0(X_t))\Pi_t)\,dt + \sigma(X_t) d\hat W_t\\
d\Pi_t=\omega(X_t) \Pi_t(1-\Pi_t)\,d\hat W_t
,\end{array}\right.
\end{equation}
where $\omega(x):=\left(\mu_1(x)-\mu_0(x)\right)/\sigma(x)$ is the {\em signal-to-noise ratio} and 
\[\hat W_t:=\int_0^t\frac{d X_t}{\sigma(X_s)}
-\int_0^t\frac{1}{\sigma(X_t)}\left(\mu_0 (X_s)+\left(\mu_1(X_s)-\mu_0(X_s)\right)\Pi_s\right)\,ds\]
is the so-called innovations process; by P. Levy's theorem, $\hat W$ is a $\P_\pi$-Brownian motion. 
Note that it is clear from the representation \eqref{XPi} that the pair $(X,\Pi)$ is a 
Markov process.
Moreover, using Ito's formula, it is straightforward to check that the likelihood ratio process $\Phi$ in \eqref{Phi} satisfies 
\begin{equation}
\label{Phi1}
d\Phi_t=\omega(X_t)\Phi_t(\omega(X_t)\Pi_t\,dt + d\hat W_t).
\end{equation}

\section{A measure change}\label{measure}

In the current section, we provide a measure change which decouples $X$ from $\Pi$.
This specific measure change technique was first used in \cite{K}, and has afterwards been applied by several authors (see \cite{da}, \cite{EL}, \cite{EVa}, \cite{JP}). 

\begin{lemma}
For $t\in[0,\infty)$, denote by $\P_{\pi,t}$ the measure $\mathbb P_\pi$ restricted to $\mathcal F_t$, $\pi\in[0,1]$. We then have 
\[\frac{d\P_{\pi,t}}{d\P_{0,t}}=\frac{1+\varphi}{1+\Phi_t}.\]
\end{lemma}

\begin{proof}
From \eqref{Ppi} we have
\begin{eqnarray*}
1-\Pi_t &=& \P_\pi(\theta=0\vert \F^X_t)=(1-\pi)\P_0(\theta=0\vert\F^X_t) \frac{d\P_{\pi,t}}{d\P_{0,t}}+ \pi\P_1(\theta=0\vert\F^X_t)\frac{d\P_{\pi,t}}{d\P_{1,t}}\\
&=& (1-\pi)\frac{d\P_{\pi,t}}{d\P_{0,t}}.
\end{eqnarray*}
Therefore, 
\[\frac{1+\varphi}{1+\Phi_t}=\frac{1-\Pi_t}{1-\pi}=\frac{d\P_{\pi,t}}{d\P_{0,t}} .\]
\end{proof}

Since $1-\Pi_\tau=1/(1+\Phi_t)$ and $\Pi_t=\frac{\Phi_t}{1+\Phi_t}$,
it is now clear that
\begin{eqnarray*}
V &=& \frac{1}{1+\varphi} \sup_{\tau\in\mathcal T^X}\E^0_\varphi\Big[ 
g_0(\tau,X_\tau)F_0(\tau) + g_1(\tau,X_\tau)\Phi_\tau F_1(\tau)  \\
&&\hspace{25mm}
-\int_0^\tau(h_0(\tau,X_t)F'_0(t) + h_1(t,X_t)\Phi_t F'_1(t) ) dt
\Big],
\end{eqnarray*}
where the Markov process $(X,\Phi)$ under $\mathbb P^0$ satisfies 
\begin{equation}\label{XPhi}
\left\{\begin{array}{ll}
dX_t=\mu_0(X_t)\,dt + \sigma(X_t) \,d W_t\\
d\Phi_t=\omega(X_t)\Phi_t \,d W_t\\
\Phi_0=\varphi\end{array}\right.
\end{equation}
(cf. \eqref{X} and \eqref{Phi}).

Next we introduce the process 
\begin{equation}
\label{Phi0}
\Phi^\circ_t:=\frac{F_1(t)}{F_0(t)}\Phi_t,
\end{equation}
so that 
\begin{eqnarray}
\label{newosp}
V &=& \frac{1}{1+\varphi}\sup_{\tau\in\mathcal T^X}\E^0_\varphi\Big[ F_0(\tau)\left( g_0(\tau,X_\tau) + g_1(\tau,X_\tau)\Phi^\circ_\tau \right) \\
&& \hspace{20mm}\notag
-\int_0^\tau F_0(t)\left(  \frac{F_0'(t)}{F_0(t)}h_0(t,X_t) +\frac{F_1'(t)}{F_1(t)} h_1(t,X_t)\Phi^\circ_t \right) dt\Big].
\end{eqnarray}
Note that the process $\Phi^\circ$ satisfies 
\[d\Phi^\circ_t=\frac{f'(t)}{f(t)}\Phi^\circ_t\,dt + \omega(X_t)\Phi^\circ_t d W_t,\]
where $f(t)=F_1(t)/F_0(t)$.

\begin{remark}
	The process $\Phi^\circ$ is the likelihood ratio given observations of the processes $X$ and $1_{\{\cdot\geq \gamma\}}$
on the event $\{ \gamma>t\}$.
	Indeed, for $t\leq T$, defining
	\[
	\Pi^\circ_t := \P_\pi(\theta=1\vert \F^X_t, \gamma>t)=\frac{\P_\pi(\theta=1,\gamma>t\vert \F^X_t)}{\P_\pi(\gamma>t\vert \F^X_t)}
	= \frac{\Pi_tF_1(t)}{\Pi_tF_1(t) + (1-\Pi_t)F_0(t)},\]
	we have 
	\[\Pi^\circ_t=\frac{\Phi^\circ_t}{\Phi^\circ_t+1}.\]
\end{remark}

Based on studies of the three-dimensional Markov process $(t,X, \Phi^\circ)$, the stopping problem \eqref{newosp} can be further studied, see the examples below. 
First, however, we summarise our theoretical findings in the following theorem.

\begin{theorem}
\label{main}
Denote by 
\begin{eqnarray}
\label{v}
v &=& \sup_{\tau\in\mathcal T^X}\E^0_\varphi\Big[F_0(\tau)\left(g_0(\tau,X_\tau)+ 
g_1(\tau,X_\tau)\Phi^\circ_\tau \right)\\
&&\notag \hspace{15mm}-\int_0^\tau F_0(t)\left(  \frac{F_0'(t)}{F_0(t)}h_0(t,X_t) +\frac{F_1'(t)}{F_1(t)} h_1(t,X_t)\Phi^\circ_t \right) dt \Big],
\end{eqnarray}
where $(X,\Phi^\circ)$ is given by \eqref{XPhi} and \eqref{Phi0}. Then $V=v/(1+\varphi)$, where $\varphi=\pi/(1-\pi)$. Moreover, if $\tau\in\mathcal T^X$ is an optimal stopping in \eqref{v}, then it is also optimal in the original problem \eqref{osp}.
\end{theorem}

\section{An example: a hiring problem}
\label{hiring}

In this section we consider a (simplistic) version of a hiring problem.
To describe this, consider a situation where a company tries to decide whether or not to employ a certain candidate, where there is considerable uncertainty about the candidate's ability.  The candidate is either of a 'good type' or of a 'bad type', 
and during the employment procedure, tests are performed to find out which is the true state.
At the same time, the candidate is potentially lost for the company as he/she may receive  other offers. Moreover, the rate at which such offers are presented, may depend on the 
ability of the candidate; for example, a candidate of the good type could be more likely to be recruited to other companies than a candidate of the bad type.

To model the above hiring problem, we let $h_i\equiv 0$, $i=0,1$, and 
\[g(t,x,\theta)=\left\{\begin{array}{cl}
-e^{-rt}c & \mbox{if }\theta=0\\
e^{-rt}d & \mbox{if }\theta=1\end{array}\right.\]
where $c$ and $d$ are positive constants representing the overall cost and benefit of hiring the candidate, respectively, 
and $r> 0$ is a constant discount rate. 
To learn about the unknown state $\theta$, tests are performed and represented as a Brownian motion 
\[X_t=\mu (\theta) t+\sigma W_t\] 
with state-dependent drift 
\[\mu(\theta)=\left\{\begin{array}{ll}
\mu_0 & \mbox{if }\theta=0\\
\mu_1 & \mbox{if }\theta=1,\end{array}\right.\]
where $\mu_0<\mu_1$.
We further assume that the survival probabilities $F_0$ and $F_1$ decay exponentially in time, i.e.
\[	F_0(t) = e^{-\lambda_0 t}\quad\quad\&\quad\quad
		F_1(t) = e^{-\lambda_1 t},\]
where $\lambda_0, \lambda_1\geq 0$ are known constants. 
The stopping problem \eqref{osp} under consideration is thus
\[V=\sup_{\tau\in\mathcal T^{X,\gamma}}\E_\pi\left[ e^{-r\tau}\left(d1_{\{\theta=1\}}-c1_{\{\theta=0\}}\right)1_{\{\tau<\gamma\}}\right],\]
where $\pi=\P_{\pi}(\theta=1)$.

By Theorem~\ref{main}, we have
\[V = \frac{1}{1+\varphi}\sup_{\tau\in\mathcal T^X}\E_\varphi^0\left[ e^{-(r+\lambda_0)\tau}\left(  \Phi^\circ_\tau d -c\right) \right],\]
where the underlying process $\Phi^{\circ}$ is a geometric Brownian motion satisfying
\[d\Phi_t^\circ = -(\lambda_1-\lambda_0)\Phi_t^\circ\,dt + \omega\Phi^\circ_t\,dW.\]
%
Clearly, the value of the stopping problem is
\[V = \frac{d}{1+\varphi}\sup_{\tau\in\mathcal T^X}\E_\varphi^0\left[ e^{-(r+\lambda_0)\tau}\left(  \Phi^\circ_\tau-\frac{c}{d}  \right) \right]= \frac{d}{1+\varphi} V^{Am}(\varphi),\]
where $V^{Am}$ is the value of the American call option with underlying $\Phi^{\circ}$ and strike $\frac{c}{d}$. 
Standard stopping theory gives that the corresponding value function is 
	\[V =\begin{cases} \frac{db^{1-\gamma}}{\gamma(1+\varphi)}{\varphi}^\gamma,\enskip \varphi<b,\\
	 \frac{d}{1+\varphi}(\varphi-\frac{c}{d}),\enskip \varphi\geq b,	\end{cases}\] 
	where $\gamma>1$ is the positive solution of the quadratic equation
	\[\frac{\omega^2}{2}\gamma(\gamma-1)+(\lambda_0-\lambda_1)\gamma-(r+\lambda_0)=0,\]
	and $b = \frac{c \gamma}{d(\gamma-1)}$. Furthermore, 
\[\tau:=\inf\{t\geq 0:\Phi^\circ_t\geq b\}\]
is an optimal stopping time. More explicitly, in terms of the process $X$ we have 
\[\tau=\inf\left\{t\geq 0:X_t\geq x + \frac{\sigma}{\omega}\left(\ln\left(\frac{b}{\varphi}\right) +(\lambda_1-\lambda_0)t\right)
+\frac{\mu_0+\mu_1}{2}t\right\},\]
where $\omega:=(\mu_1-\mu_0)/\sigma$.


\section{An example: closing a short position}

In this section we study an example of optimal closing of a short position under recall risk, cf. \cite{GH}.
We consider a short position in an underlying stock with unknown drift, where 
the random horizon corresponds to a time point at which the counterparty recalls the position. 
Naturally, the counterparty favours a large drift, so the risk of recall is greater in the state with a small drift.
A similar model (but with no recall risk) was studied in \cite{EL}.

Let the stock price be modeled by 
geometric Brownian motion with dynamics
\[dX_t = \mu(\theta) X_t\,dt+ \sigma X_t \,dW_t,\]
where the drift is state-dependent with $\mu(0)=\mu_0 <\mu_1=\mu(1)$, and $\sigma$ is a known constant. 
We let $g(t,x,\theta)=h(t,x,\theta)=xe^{-rt}$ and consider the stopping problem
\[V=\inf_{\tau\in\mathcal T^{X,\gamma}}\E_\pi[e^{-r\tau\wedge\gamma} X_{\tau\wedge\gamma}],\]
where 
\[F_i(t):=\P_\pi(\gamma>t\vert\theta=i) = e^{-\lambda_i t},\]
with $\lambda_0>0=\lambda_1$ and
\[\P_\pi(\theta= 1)=\pi =1-\P_\pi(\theta = 0) .\]
Here $r$ is a constant discount rate; to avoid degenerate cases, we assume that $r\in(\mu_0,\mu_1)$.

Then the value function can be written as $V=v/(1+\varphi)$, where
\begin{equation}
\label{ueq}
v = \inf_{\tau\in\mathcal T^X}\E^0_\varphi\Big[ e^{-r\tau}F_0(\tau)X_{\tau}\left( 1+ \Phi^\circ_\tau \right) + \lambda_0\int_0^\tau e^{-rt}F_0(t) X_{t}\left( 1+ \Phi^\circ_t  \right)\,dt\Big],
\end{equation}
with $d\Phi^\circ_t = -\lambda_1\Phi^\circ_t dt + \omega \Phi^\circ_t\,dW_t$ and $\Phi^\circ_0 = \varphi$. Here  $\omega = \frac{\mu_1-\mu_0}{\sigma}$.

Another change of measure will remove the occurrencies of $X$ in \eqref{ueq}. In fact, let $\tilde \P$ be a measure with 
\[\left.\frac{d\tilde \P}{d\P^0}\right\vert_{\F_t} = e^{-\frac{\sigma^2}{2}t + \sigma W_t},\]
so that $\tilde W_t=-\sigma t+W_t$ is a $\tilde\P-$Brownian motion. Then 
\begin{equation}
\label{ufin}
v=x\inf_{\tau\in\mathcal T^X}\tilde\E_\varphi\Big[ e^{-(r+\lambda_0-\mu_0)\tau}\left( 1+ \Phi^\circ_\tau \right) + \lambda_0\int_0^\tau e^{-(r+\lambda_0-\mu_0)t} \left( 1+ \Phi^\circ_t  \right)\,dt\Big],
\end{equation}
with 
\[d\Phi^0_t=(\lambda_0 +\sigma\omega)\Phi^0_t\,dt + \omega \Phi^0_t\,d\tilde W.\]

The optimal stopping problem \eqref{ufin} is a one-dimensional time-homogeneous problem, and is thus straightforward to analyze using standard stopping theory. Indeed, setting
\[\bar v:=\frac{v}{x}= \inf_{\tau\in\mathcal T^X}\tilde\E_\varphi\Big[ e^{-(r+\lambda_0-\mu_0)\tau}\left( 1+ \Phi^\circ_\tau \right) + \lambda_0\int_0^\tau e^{-(r+\lambda_0-\mu_0)t} \left( 1+ \Phi^\circ_t  \right)\,dt\Big],\]
the associated free-boundary problem is to find
$(\bar v,B)$ such that 
\begin{equation}\label{fbpu}
\left\{\begin{array}{rl}
\frac{\omega^2\varphi^2}{2}\bar v_{\varphi\varphi} + (\lambda_0+\mu_1-\mu_0)\varphi\bar v_\varphi-
(r+\lambda_0-\mu_0)\bar v + \lambda_0(1+\varphi)=0 & \varphi<B\\
\bar v(\varphi)=1+\varphi &\varphi\geq B\\
\bar v_\varphi(B)=1,\end{array}\right.
\end{equation}
and such that $\bar v\leq 1+\varphi$. Solving the free-boundary problem \eqref{fbpu} gives
\[B=\frac{\gamma(r-\mu_0)(\mu_1-r)}{(1-\gamma)(r+\lambda_0-\mu_0)(\lambda_0+\mu_1-r)}\]
and
\[\bar v(\varphi)=\left\{\begin{array}{ll}
\frac{r-\mu_0}{(1-\gamma)(r+\lambda_0-\mu_0)}\left(\frac{\varphi}{B}\right)^\gamma-\frac{\lambda_0}{\mu_1-r}\varphi +\frac{\lambda_0}{r+\lambda_0-\mu_0} & \varphi <B\\
1+\varphi & \varphi\geq B,
\end{array}\right.\]
where $\gamma$ is the positive solution of the quadratic equation
\[\frac{\omega^2}{2}\gamma(\gamma-1)+(\lambda_0+\mu_1-\mu_0)\gamma-(r+\lambda_0-\mu_0)=0.\]
A standard verification argument then gives that $V=\frac{x}{1+\varphi}\bar v(\varphi)$, and 
\[\tau_B:=\inf\{t\geq 0:\Phi^\circ_t\geq B\}= \inf\{t\geq 0:\Phi_t\geq Be^{-\lambda_0t}\}\]
is optimal in \eqref{ueq}.

\section{An example: a sequential testing problem with a random horizon}
\label{testing}
Consider the sequential testing problem for a Wiener process, i.e. the problem of determining
as quickly, and accurately,  the unknown drift $\theta$ from observations of the process
\[X_t=\theta t + \sigma W_t.\]
Similar to the classical version (see \cite{S}), we assume that $\theta$ is Bernoulli distributed with $\P(\theta=1)=\pi=1-\P(\theta=0)$,
where $\pi\in(0,1)$. In \cite{NP}, the sequential testing problem has been studied under a random horizon. 
Here we consider an instance of a testing problem which further extends the set-up 
by allowing the distribution of the random horizon to depend on the unknown state.

More specifically, we assume that when $\theta = 1$, then the horizon $\gamma$ is infinite, i.e. 
$F_1(t) = 1$ for all $t$; and when $\theta = 0$, the time horizon is exponentially distributed with rate $\lambda$, 
i.e. $F_0(t) = e^{-\lambda t}$. Mimicking the classical formulation of the problem, we study the 
problem of minimizing 
\[\P(\theta\not= d) + c\E[\tau]\]
over all stopping times $\tau\in\mathcal T^{X,\gamma}$ and $\F^{X,\gamma}_\tau$-measurable decision rules $d$ with values in $\{0,1\}$.
By standard methods, it is clear that the optimization problem reduces to a stopping problem
\[V=\inf_{\tau\in\mathcal T^{X,\gamma}}\E_\pi\left[\hat\Pi_\tau\wedge (1-\hat\Pi_\tau) + c\tau\right],\]
where 
\[\hat\Pi_t:=\P_\pi(\theta=1\vert \F^{X,\gamma}_t).\]
Moreover, the process $\hat\Pi$ satisfies 
\[\hat\Pi_t=\left\{\begin{array}{cl} \Pi^\circ_t & t<\gamma\\
	0 & t\geq \gamma,\end{array}\right.\]
where 
\[\Pi^\circ_t=\frac{\Pi_t}{\Pi_t + (1-\Pi_t)e^{-\lambda t}}=\frac{\P_\pi(\theta=1\vert \F^{X}_t)}{\P_\pi(\theta=1\vert \F^{X}_t) + (1-\P_\pi(\theta=1\vert \F^{X}_t))e^{-\lambda t}},\]
and it follows that 
\begin{eqnarray*}
	V &=& \inf_{\tau\in\mathcal T^{X,\gamma}}\E_\pi\left[\hat\Pi_\tau\wedge (1-\hat\Pi_\tau) + c\tau\right]\\
	&=& \inf_{\substack{\tau\in\mathcal T^{X,\gamma}\\
			\tau\leq \gamma}}\E_\pi\left[\hat\Pi_\tau\wedge (1-\hat\Pi_\tau) + c\tau\right]\\
	&=& \inf_{\tau\in\mathcal T^{X,\gamma}
	}\E_\pi\left[\left(\Pi^\circ_\tau\wedge (1-\Pi_\tau^\circ)\right)1_{\{\tau<\gamma\}} + c\int_0^\tau 1_{\{t< \gamma\}}\,dt\right].
\end{eqnarray*}
Following the general methodology before Theorem~\ref{main}, we find that 
\begin{eqnarray*}
	V &=& \inf_{\tau\in\mathcal T^{X}
	}\E_\pi\left[\left(\Pi^\circ_\tau\wedge (1-\Pi_\tau^\circ)\right)1_{\{\tau<\gamma\}} + c\int_0^\tau 1_{\{t< \gamma\}}\,dt\right]\\
	&=& \inf_{\tau\in\mathcal T^{X}
	}\E_\pi\left[\left(\Pi^\circ_\tau\wedge (1-\Pi_\tau^\circ)\right)((1-\Pi_\tau)F_0(\tau) +\Pi_\tau)+ c\int_0^\tau 
	((1-\Pi_t)F_0(t) +\Pi_t)
	\,dt\right]\\
	&=& \frac{1}{1+\varphi} \inf_{\tau\in\mathcal T^{X}}
	\E^0_\varphi\left[F_0(\tau)\left(\Phi^\circ_\tau\wedge 1\right)+ c\int_0^\tau F_0(t)(1 +\Phi^\circ_t) \,dt\right].
\end{eqnarray*}
Here $\Phi^\circ:=\Pi^\circ/(1-\Pi^\circ)$ satisfies 
\[d\Phi^\circ_t=\lambda \Phi^\circ_t\,dt + \omega \Phi^\circ_t d W_t,\]
where $\omega = \frac{1}{\sigma}$. 

Standard stopping theory can now be applied to solve the sequential testing problem with a random horizon.
Setting 
\[v(\varphi):=
\inf_{\tau\in\mathcal T^{X}}
\E_{\varphi}^0\left[F_0(\tau)\left(\Phi^\circ_\tau\wedge 1\right)+ c\int_0^\tau F_0(t)(1 +\Phi^\circ_t) \,dt\right]\]
one expects a two-sided stopping region $(0,A]\cup[B,\infty)$, and $v$ to satisfy 
%
%
\[\left\{\begin{array}{rl}
	\frac{1}{2}\omega^2\varphi^2 v_{\varphi\varphi}+\lambda \varphi v_\varphi -\lambda v+c(1+\varphi) =0, & \varphi \in (A,B)\\
	v(A) =A\\
	v_\varphi(A) =1\\
	v(B) =1\\
	v_\varphi(B) =0
\end{array}\right.\]
for some constants $A,B$ with $0<A<1<B$. The general solution of the ODE is easily seen to be
\[v(\varphi)=C_1 \varphi^{-\frac{2\lambda}{\omega^2}}+C_2 \varphi+\frac{c}{\lambda}  - \frac{c}{\lambda+\frac{1}{2}\omega^2}\varphi\ln(\varphi)\]
where $C_1, C_2$ are arbitrary constants. 
Since the stopping region is two-sided, explicit solutions are not expected. Instead,
using the four boundary conditions, equations for the unknowns $C_1$, $C_2$, $A$ and $B$ can be derived using standard methods; we omit the details.

\bibliographystyle{abbrv}
\bibliography{references}

\end{document}